\documentclass[11pt]{amsart}

\usepackage[T1]{fontenc}
\usepackage{amssymb,amsthm,mathtools}
\usepackage[expansion=false]{microtype}
\usepackage{xcolor}
\usepackage[
  colorlinks=true,
  linkcolor=blue!55!black,
  citecolor=green!45!black,
  urlcolor=blue!65!black
]{hyperref}

\numberwithin{equation}{section}
\mathtoolsset{showonlyrefs}

\theoremstyle{plain}
\newtheorem{theorem}{Theorem}[section]
\newtheorem{lemma}[theorem]{Lemma}
\newtheorem{proposition}[theorem]{Proposition}

\theoremstyle{definition}

\newcommand{\R}{\mathbb{R}}
\newcommand{\Z}{\mathbb{Z}}
\newcommand{\N}{\mathbb{N}}
\newcommand{\cS}{\mathcal{S}}

\title{Sharp Bounds for Discrete Cube Skeletons}
\author{Dean Menezes}
\address{Department of Mathematics, University of Texas at Austin, Austin, TX 78712, USA}
\email{dean.menezes@utexas.edu}
\date{August 25, 2026}
\hypersetup{
  pdftitle={Sharp Bounds for Discrete Cube Skeletons},
  pdfauthor={Dean Menezes},
  pdfsubject={Sharp finite-cardinality bounds for discrete cube skeletons},
  pdfkeywords={cube skeleton, discrete geometry, entropy, Shearer's inequality}
}

\begin{document}

\begin{abstract}
Fix integers $0\leq k<n$.  Let $F_{n,k}(N)$ be the least size of a finite
set $B\subset\Z^n$ that contains a filled axis-parallel cube $k$-skeleton
centered at each point of some $N$-point set.  We prove that $F_{n,k}(N)$ has
order $N^{1-(n-k)/(2n^2)}$, with constants depending only on $n$ and $k$.
Thornton proved the upper bound and lower bounds with every smaller exponent;
the endpoint lower bound was open for $k\geq1$.  For square boundaries in
$\Z^2$, the exponent is $7/8$.  A midpoint count and Shearer's inequality
handle large radii; induction in lattice cells handles small radii.
\end{abstract}

\maketitle

\section{Introduction}

Fix $0\leq k<n$.  Given a finite set $S\subset\Z^n$, choose a positive integer
radius $r_x$ for each $x\in S$.  We ask how small a set $B\subset\Z^n$ can be
if it contains the filled $k$-skeleton of the cube centered at each $x$.

Write $[n]=\{1,\ldots,n\}$ and $\N=\{1,2,\ldots\}$.  For a finite set $E$,
let $\binom{E}{m}$ be the family of its $m$-element subsets.  For $x\in\Z^n$,
$r\in\N$, and $0\leq k<n$, set
\[
 \cS_k(x,r)=
 \bigcup_{J\in\binom{[n]}{k}}
 \left\{
 y\in\Z^n:
 \begin{array}{ll}
 |y_j-x_j|\leq r,&j\in J,\\
 |y_i-x_i|=r,&i\notin J
 \end{array}
 \right\}.
\]
The coordinates in $J$ are free; the other $n-k$ coordinates are fixed at
$x_i\pm r$.  Thus $\cS_0(x,r)$ is the vertex set, while
$\cS_{n-1}(x,r)$ is the lattice boundary of the cube.

Let $F_{n,k}(N)$ be the minimum of $|B|$ over finite sets
$B,S\subset\Z^n$ with $|S|=N$ and radii $r_x\in\N$ such that
\[
 \cS_k(x,r_x)\subset B\qquad(x\in S).
\]

\begin{theorem}\label{thm:main}
For each $0\leq k<n$, there are constants $c_{n,k},C_{n,k}>0$ such that
\[
 c_{n,k}N^{1-(n-k)/(2n^2)}
 \leq F_{n,k}(N)
 \leq C_{n,k}N^{1-(n-k)/(2n^2)}
\]
for every $N\geq1$.
\end{theorem}

Thornton proved the upper bound and lower bounds with every smaller exponent
\cite[Theorems~1.3 and~1.5]{Thornton2017}.  For square boundaries, Keleti,
Nagy, and Shmerkin proved $(N/\log N)^{7/8}$ and asked whether the logarithm
could be removed \cite[Theorem~1.7]{KeletiNagyShmerkin2018}.
Theorem~\ref{thm:main} removes it.  For $k=0$, the lower bound also follows
from Thornton's orthoplex theorem after an invertible linear change of
coordinates \cite[Theorem~1.9(2)]{Thornton2017}; our proof treats all $k$ at
once.

In particular, suppose that $B,S\subset\Z^2$ are finite and that for each
$x\in S$ there is an $r_x\in\N$ with $\cS_1(x,r_x)\subset B$.  Then
\[
 |B|\geq c|S|^{7/8}
\]
for an absolute constant $c>0$, and the exponent is sharp.  Here each
skeleton is the boundary of a square and has $8r_x$ lattice points.

For the lower bound, put $\ell=n-k$.  A midpoint count first
gives many cube vertices.  Shearer's inequality then gives at least a constant
multiple of $N^{\ell(2n-1)/(2n^2)}$ distinct tuples of fixed coordinates.
Split the centers at radius $L=aN^{1/n}$, where $a>0$ is a small constant.
If at least half the radii are at least $L$, one choice of the fixed coordinate
positions gives many disjoint $k$-faces, each with at least $L^k$ points.  If
at least half are smaller than $L$, divide $\Z^n$ into cells of side
$\lceil L\rceil$ and use induction on the number of centers.  A point belongs
to at most $3^n$ of the resulting cellwise unions.  Thornton's digit
construction gives the matching upper bound.

\section{Counting lemmas}

Every random vector below has finite support, and all logarithms are natural.
For entropy notation, see Cover and Thomas \cite{CoverThomas2006}.  If
$I=\{i_1<\cdots<i_m\}\subset[n]$, write
$X_I=(X_{i_1},\ldots,X_{i_m})$.

We use Shearer's inequality \cite{ChungGrahamFranklShearer1986}.  Applied to
the $n$ sets $[n]\setminus\{i\}$, it gives the finite Loomis--Whitney
inequality \cite{LoomisWhitney1949}.

\begin{lemma}[Shearer's inequality]\label{lem:shearer}
Let $X=(X_1,\ldots,X_n)$ be a random vector.  Let $\mathcal I$ be a family of
subsets of $[n]$, and suppose that each coordinate belongs to at least $q$
members of $\mathcal I$.  Then
\[
 qH(X)\leq\sum_{I\in\mathcal I}H(X_I).
\]
\end{lemma}

\begin{proof}
For each $I\in\mathcal I$, the chain rule and the fact that conditioning cannot increase
entropy give
\[
 H(X_I)
 =\sum_{j\in I}H\bigl(X_j\mid X_{I\cap[j-1]}\bigr)
 \geq\sum_{j\in I}H\bigl(X_j\mid X_{[j-1]}\bigr).
\]
After summing over $I$, each term on the right occurs at least $q$ times.
The chain rule for $H(X)$ gives the result.
\end{proof}

\begin{lemma}[Coordinate midpoints]\label{lem:axial}
Let $C,T\subset\R^n$ be finite.  Suppose that for every $x\in T$ and $i\in[n]$,
there is an $r_{x,i}>0$ such that
\[
 x-r_{x,i}e_i,\ x+r_{x,i}e_i\in C,
\]
where $e_i$ is the $i$th coordinate vector.  Then
\[
 |C|\geq\sqrt2\,|T|^{(2n-1)/(2n)}.
\]
\end{lemma}

\begin{proof}
Assume $T\neq\varnothing$.  Fix $i$ and partition $\R^n$ into lines parallel
to $e_i$.  For each line $L$, let
\[
 d_L=|T\cap L|,
 \qquad
 b_L=|C\cap L|.
\]
The points of $T\cap L$ are distinct midpoints of pairs from $C\cap L$.
Hence
\[
 d_L\leq\binom{b_L}{2}\leq\frac{b_L^2}{2},
 \qquad
 |C|\geq\sqrt2\sum_L\sqrt{d_L}.
\]

Let $X$ be uniform on $T$, and put $p_L=d_L/|T|$ when $d_L>0$.  In the
next display, the sums and product run over these lines.  The projection
$X_{[n]\setminus\{i\}}$ has probabilities $p_L$, so weighted AM--GM gives
\[
 \sum_L\sqrt{p_L}
 =\sum_L p_L\,p_L^{-1/2}
 \geq\prod_L p_L^{-p_L/2}
 =\exp\left(\frac12H(X_{[n]\setminus\{i\}})\right).
\]
The preceding estimate holds for every $i$.  By Lemma~\ref{lem:shearer},
choose $i$ such that
\[
 H(X_{[n]\setminus\{i\}})
 \geq\frac{n-1}{n}H(X)
 =\frac{n-1}{n}\log|T|.
\]
For this $i$,
\[
 |C|
 \geq\sqrt{2|T|}\sum_L\sqrt{p_L}
 \geq\sqrt2\,|T|^{(2n-1)/(2n)}.
\]
\end{proof}

For $x\in\R^n$ and $r>0$, let
\[
 \mathcal V(x,r)=x+r\{\pm1\}^n
\]
be the vertex set of the axis-parallel cube with center $x$ and half-side $r$.

\begin{lemma}[Cube vertices]\label{lem:vertices}
Let $C,T\subset\R^n$ be finite.  Suppose that for every $x\in T$, there is an
$r_x>0$ such that $\mathcal V(x,r_x)\subset C$.  Then
\[
 |C|\geq\sqrt2\,|T|^{(2n-1)/(2n)}.
\]
\end{lemma}

\begin{proof}
Set
\[
 v_1=(1,\ldots,1),
 \qquad
 v_i=v_1-2e_i\quad(2\leq i\leq n).
\]
Since $v_i-v_1=-2e_i$ for $i\geq2$, the span of the $v_i$ contains
$e_2,\ldots,e_n$.  It also contains
$e_1=v_1-e_2-\cdots-e_n$.  Thus the $v_i$ form a basis.  Let $M$ be the
linear map with $Me_i=v_i$.  Each $v_i$ is a cube vertex, so $C$ contains
$x\pm r_xv_i$ for every $x\in T$ and $i\in[n]$.
Thus $M^{-1}C$ contains
\[
 M^{-1}x\pm r_xe_i\qquad(i\in[n]).
\]
Apply Lemma~\ref{lem:axial} to $M^{-1}C$ and $M^{-1}T$.
\end{proof}

\section{Fixed-coordinate tuples}

A $k$-face has $k$ free coordinates and $\ell=n-k$ fixed coordinates.  The
next lemma counts the possible fixed-coordinate tuples.

\begin{lemma}[Fixed-coordinate bound]\label{lem:fixed}
Fix $1\leq\ell\leq n$.  Let $A\subset\R^\ell$ and $T\subset\R^n$ be finite.
Suppose that for every $x\in T$, there is an $r_x>0$ such that
\[
 x_I+r_x\sigma\in A
\]
for every $I\in\binom{[n]}{\ell}$ and $\sigma\in\{\pm1\}^\ell$.  Then
\[
 |A|\geq
 2^{\ell/(2n)}|T|^{\ell(2n-1)/(2n^2)}.
\]
\end{lemma}

\begin{proof}
Assume $T\neq\varnothing$, and set
\[
 C=\left\{z\in\R^n:z_I\in A
 \text{ for every }I\in\binom{[n]}{\ell}\right\}.
\]
The set $C$ is finite.  For each coordinate $j$, choose an $\ell$-set $I$
containing $j$.  Since $A$ is finite, $z_j$ has only finitely many choices.

Fix $x\in T$ and $\tau\in\{\pm1\}^n$.  For every $I$,
\[
 (x+r_x\tau)_I=x_I+r_x\tau_I\in A.
\]
Thus $\mathcal V(x,r_x)\subset C$, and Lemma~\ref{lem:vertices} gives
\[
 |C|\geq\sqrt2\,|T|^{(2n-1)/(2n)}.
\]

Let $Z$ be uniform on $C$.  Each coordinate belongs to
$\binom{n-1}{\ell-1}$ of the $\ell$-subsets of $[n]$, and each $Z_I$ takes
values in $A$.  Lemma~\ref{lem:shearer} gives
\[
 \binom{n-1}{\ell-1}\log|C|
 \leq\sum_{|I|=\ell}H(Z_I)
 \leq\binom{n}{\ell}\log|A|.
\]
Since
$\binom{n-1}{\ell-1}/\binom{n}{\ell}=\ell/n$,
\[
 |A|\geq|C|^{\ell/n}
 \geq2^{\ell/(2n)}|T|^{\ell(2n-1)/(2n^2)}.
\]
\end{proof}

\section{The lower bound}

Set
\[
 \ell=n-k,
 \qquad
 \gamma=\frac{\ell(2n-1)}{2n^2},
 \qquad
 \beta=\gamma+\frac{k}{n}=1-\frac{\ell}{2n^2}<1.
\]

\begin{proposition}\label{prop:lower}
For each $0\leq k<n$, there is a constant $c_{n,k}>0$ such that
\[
 |B|\geq c_{n,k}|S|^\beta
\]
whenever finite sets $B,S\subset\Z^n$ and radii $r_x\in\N$ satisfy
$\cS_k(x,r_x)\subset B$ for every $x\in S$.
\end{proposition}

\begin{proof}
Choose $a\in(0,1/2)$ so small that, with $\rho=(2a)^n$,
\[
 2\cdot3^n\rho^{1-\beta}\leq1.
\]
Then $\rho<1$.  Set
\[
 d=\binom{n}{\ell}^{-1}2^{\ell/(2n)-\gamma},
 \qquad
 c_{n,k}=\min\{a^{n\beta},da^k\}.
\]
We prove the claim by strong induction on $N=|S|$.  The case $N=0$ is
trivial.
Put $L=aN^{1/n}$.  If $L<1$, then each skeleton is nonempty and
\[
 |B|\geq1>L^{n\beta}=a^{n\beta}N^\beta\geq c_{n,k}N^\beta.
\]
Assume now that $L\geq1$ and that the claim holds for fewer than $N$ centers.
Call $r_x$ large when $r_x\geq L$ and small otherwise.  At least one class has
at least $N/2$ centers.

Suppose first that the large class $S_{\mathrm L}$ has at least $N/2$ centers.
For $I\in\binom{[n]}{\ell}$, let
\[
 A_I=\{x_I+r_x\sigma:x\in S_{\mathrm L},\
       \ \sigma\in\{\pm1\}^\ell\},
 \qquad
 A=\bigcup_{|I|=\ell}A_I.
\]
Lemma~\ref{lem:fixed} and averaging give an $I$ such that
\[
 |A_I|\geq dN^\gamma.
\]
For each $u\in A_I$, choose $x\in S_{\mathrm L}$ and
$\sigma\in\{\pm1\}^\ell$ with $u=x_I+r_x\sigma$.  With
$J=[n]\setminus I$, the skeleton centered at $x$ contains
\[
 Q_u=\{y\in\Z^n:y_I=u,\ |y_j-x_j|\leq r_x\text{ for }j\in J\}.
\]
The sets $Q_u$ are disjoint because their $I$-coordinate tuples differ.  Each
has $(2r_x+1)^k\geq L^k$ points.  Therefore
\[
 |B|\geq |A_I|L^k
 \geq da^kN^{\gamma+k/n}
 \geq c_{n,k}N^\beta.
\]

Suppose instead that the small class $S_{\mathrm s}$ has at least $N/2$
centers.  Let $h=\lceil L\rceil$ and partition $\Z^n$ into the cells
\[
 P_m=hm+\{0,1,\ldots,h-1\}^n,
 \qquad m\in\Z^n.
\]
Set
\[
 S_m=S_{\mathrm s}\cap P_m,
 \qquad
 N_m=|S_m|,
 \qquad
 Y_m=\bigcup_{x\in S_m}\cS_k(x,r_x).
\]
Since $L\geq1$, we have $h\leq2L$, and hence
\[
 N_m\leq h^n\leq(2L)^n=\rho N<N.
\]
The induction hypothesis gives $|Y_m|\geq c_{n,k}N_m^\beta$ whenever
$N_m>0$.

Each point lies in at most $3^n$ sets $Y_m$.  Indeed, if $y\in Y_m$, then
$|y_i-x_i|<h$ for some $x\in P_m$ and every $i$; for fixed $y_i$, at most
three cells are possible in that coordinate.  Since $\bigcup_m Y_m\subset B$,
\[
 3^n|B|\geq\sum_m |Y_m|
 \geq c_{n,k}\sum_m N_m^\beta.
\]
Since $N_m\leq\rho N$, $\beta<1$, and
$\sum_m N_m=|S_{\mathrm s}|\geq N/2$,
\[
 \sum_m N_m^\beta
 \geq(\rho N)^{\beta-1}\sum_m N_m
 \geq\frac{N^\beta}{2\rho^{1-\beta}}.
\]
Therefore
\[
 |B|\geq
 \frac{c_{n,k}}{2\cdot3^n\rho^{1-\beta}}N^\beta
 \geq c_{n,k}N^\beta.
\]
This completes the induction.
\end{proof}

\section{The upper bound}

We use Thornton's digit construction
\cite[Lemma~2.1 and Theorem~1.5]{Thornton2017}.  We enlarge the interval for
each free coordinate so that it contains every integer from $x_j-r_x$ to
$x_j+r_x$.  This changes only the constant.

\begin{lemma}[Digit construction]\label{lem:digits}
For every $h\geq2$, there is a set $D_h\subset\Z$ with
\[
 |D_h|\leq C_n h^{2n-1}
\]
such that every $x\in\{1,\ldots,h^{2n}-1\}^n$ has an integer $r_x$ with
\[
 1\leq r_x<h^{2n}
 \qquad\text{and}\qquad
 x_j-r_x,x_j+r_x\in D_h\quad(j\in[n]).
\]
\end{lemma}

\begin{proof}
Let
\[
 D_h=
 \left\{
 \sum_{q=0}^{2n-1}a_q h^q:
 -2(h-1)\leq a_q\leq2(h-1),\
 \ \prod_{q=0}^{2n-1}a_q=0
 \right\}.
\]
There are at most $2n(4h-3)^{2n-1}$ such coefficient vectors: choose a
zero position, then choose the other coefficients.  Hence
\[
 |D_h|\leq2n(4h-3)^{2n-1}\leq C_n h^{2n-1}.
\]

Write
\[
 x_j=\sum_{q=0}^{2n-1}x_{j,q}h^q,
 \qquad 0\leq x_{j,q}<h.
\]
Choose $q_0$ with $x_{1,q_0}>0$.  Pair the digit positions as
$\{0,1\},\{2,3\},\ldots,\{2n-2,2n-1\}$.  Choose a bijection
$\pi:[n]\to\{0,\ldots,n-1\}$ such that
$\pi(1)=\lfloor q_0/2\rfloor$, and set
\[
 r_0=
 \sum_{j=1}^n x_{j,2\pi(j)}h^{2\pi(j)}
 -\sum_{j=1}^n x_{j,2\pi(j)+1}h^{2\pi(j)+1}.
\]
The pair assigned to coordinate $1$ contains $q_0$, so at least one selected
digit is nonzero.  Let $q$ be the highest position with a nonzero selected
digit.  Its term has absolute value at least $h^q$, while
all lower terms together have absolute value at most
\[
 (h-1)\sum_{p<q}h^p=h^q-1.
\]
Thus $r_0\neq0$.  Also,
\[
 |r_0|\leq\sum_{q=0}^{2n-1}(h-1)h^q=h^{2n}-1.
\]

Fix $j$.  In the displayed signed representation of $x_j-r_0$, the
coefficient of $h^{2\pi(j)}$ is zero.  In that of $x_j+r_0$, the coefficient
of $h^{2\pi(j)+1}$ is zero.  Every coefficient lies between $-2(h-1)$ and
$2(h-1)$.  Hence both numbers lie in $D_h$.  No carrying is needed because
membership in $D_h$ asks only for such a signed representation.  Set
$r_x=|r_0|$; changing the sign of $r_0$ only swaps the two endpoints.
\end{proof}

Recall that $\beta=1-(n-k)/(2n^2)$.

\begin{proposition}\label{prop:upper}
For each $0\leq k<n$, there is a constant $C_{n,k}>0$ such that
\[
 F_{n,k}(N)\leq C_{n,k}N^\beta
\]
for every $N\geq1$.
\end{proposition}

\begin{proof}
Given $N$, set
\[
 h=\left\lceil (N^{1/n}+1)^{1/(2n)}\right\rceil.
\]
Then $h\geq2$ and
\[
 N\leq(h^{2n}-1)^n,
 \qquad
 h^{2n^2}\leq2^{2n^2+n}N.
\]
Indeed, the first inequality follows from the definition of $h$.  For the
second, use $h\leq2(N^{1/n}+1)^{1/(2n)}$ and
$N^{1/n}+1\leq2N^{1/n}$.

Let
\[
 T_h=\{1,\ldots,h^{2n}-1\},
 \qquad
 I_h=[-h^{2n},2h^{2n}]\cap\Z.
\]
Set
\[
 B_h=
 \bigcup_{J\in\binom{[n]}{k}}
 \{y\in\Z^n:y_j\in I_h\ (j\in J),\ y_i\in D_h\ (i\notin J)\}.
\]

Take $x\in T_h^n$ and choose $r_x$ by Lemma~\ref{lem:digits}.  On a $k$-face
with free-coordinate set $J$, every fixed coordinate $x_i\pm r_x$ lies in
$D_h$.  Every free coordinate lies between $x_j-r_x$ and $x_j+r_x$.  Since
$1\leq x_j,r_x<h^{2n}$, this interval lies in $I_h$.  Therefore
\[
 \cS_k(x,r_x)\subset B_h\qquad(x\in T_h^n).
\]
Since $|I_h|=3h^{2n}+1$,
\[
 |B_h|
 \leq\binom{n}{k}|D_h|^{n-k}|I_h|^k
 \leq C_{n,k}h^{(2n-1)(n-k)+2nk}
 =C_{n,k}h^{2n^2\beta}
 \leq C_{n,k}N^\beta.
\]
Finally, $|T_h^n|\geq N$, so any $N$ of its points give the required set of
centers.
\end{proof}

\begin{proof}[Proof of Theorem~\ref{thm:main}]
Combine Propositions~\ref{prop:lower} and~\ref{prop:upper}.
\end{proof}

\section{Remarks and related work}

Only the small-radius argument uses the lattice: a cell of side $h$ contains
at most $h^n$ lattice points.  A bounded cell can contain arbitrarily many
points of $\R^n$, so the same induction does not apply to unrestricted real
centers.  The fixed-coordinate bound itself holds in $\R^n$.

Keleti surveys small-union problems with large sets of centers
\cite{Keleti2017}.  Chang, Cs\"ornyei, H\'era, and Keleti treat continuous
polytope skeletons and unions of affine subspaces
\cite{ChangCsornyeiHeraKeleti2018}.  H\'era, Keleti, and M\'ath\'e prove
Hausdorff-dimension bounds for such unions and for Furstenberg-type sets
\cite{HeraKeletiMathe2019}.  Fiedler proves packing-dimension bounds for
unions and extensions of $k$-planes \cite{Fiedler2026}.  Olivo and Shmerkin
study maximal operators for cube skeletons \cite{OlivoShmerkin2020}.
Cowen-Breen, Karangozishvili, Varadarajan, and Wang study translated and
dilated patterns in arithmetic Kakeya problems \cite{CowenBreenEtAl2020}.

\bibliographystyle{amsplain}
\bibliography{references}

\end{document}